\pdfoutput=1
\RequirePackage{ifpdf}
\ifpdf % We are running pdfTeX in pdf mode
\documentclass[pdftex]{sigma}
\else
\documentclass{sigma}
\fi

\numberwithin{equation}{section}

\newtheorem{Theorem}{Theorem}[section]

\newtheorem{Lemma}[Theorem]{Lemma}
\newtheorem{Proposition}[Theorem]{Proposition}
 { \theoremstyle{definition}

\newtheorem{Remark}[Theorem]{Remark} }

\newcommand{\lieb}{\mathfrak{b}}
\newcommand{\lieg}{\mathfrak{g}}
\newcommand{\lieh}{\mathfrak{h}}
\newcommand{\liek}{\mathfrak{k}}
\newcommand{\liel}{\mathfrak{l}}
\newcommand{\liem}{\mathfrak{m}}
\newcommand{\lien}{\mathfrak{n}}
\newcommand{\liep}{\mathfrak{p}}
\newcommand{\lieu}{\mathfrak{u}}

\newcommand{\CU}{\mathbb{C}[U]}

\newcommand{\Uqg}{U_q(\mathfrak{g})}

\newcommand{\Uqt}{U_q(\mathfrak{t})}
\newcommand{\Uqu}{U_q(\mathfrak{u})}

\newcommand{\UqkS}{U_q(\mathfrak{k}_S)}
\newcommand{\UqlS}{U_q(\mathfrak{l}_S)}

\newcommand{\CqG}{\mathbb{C}_q[G]}
\newcommand{\CqU}{\mathbb{C}_q[U]}

\newcommand{\Cqflag}{\mathbb{C}_q[U / K_S]}
\newcommand{\Cqfullflag}{\mathbb{C}_q[U /T]}

\newcommand{\id}{\mathrm{id}}

\newcommand{\diff}{\mathrm{d}}
\newcommand{\del}{\partial}
\newcommand{\delbar}{\overline{\partial}}

\newcommand{\bbC}{\mathbb{C}}

\newcommand{\bbR}{\mathbb{R}}

\newcommand{\iu}{\boldsymbol{i}}

\newcommand{\weiS}{\rho_S}
\newcommand{\charS}{\xi_{\weiS}}

\newcommand{\matuni}{\mathsf{M}}
\newcommand{\proj}{\mathsf{P}}

\newcommand{\kafo}{\omega}
\newcommand{\kafot}{\tilde{\omega}}

\newcommand{\homroot}{\Delta(\lien_S^+)}

\begin{document}
\allowdisplaybreaks

\newcommand{\arXivNumber}{2003.10305}

\renewcommand{\thefootnote}{}

\renewcommand{\PaperNumber}{098}

\FirstPageHeading

\ShortArticleName{Twisted Hochschild Homology of~Quantum Flag Manifolds and K\"ahler Forms}

\ArticleName{Twisted Hochschild Homology\\ of~Quantum Flag Manifolds and K\"ahler Forms\footnote{This paper is a~contribution to the Special Issue on Noncommutative Manifolds and their Symmetries in honour of~Giovanni Landi. The full collection is available at \href{https://www.emis.de/journals/SIGMA/Landi.html}{https://www.emis.de/journals/SIGMA/Landi.html}}}

\Author{Marco MATASSA}

\AuthorNameForHeading{M.~Matassa}

\Address{OsloMet -- Oslo Metropolitan University, Oslo, Norway}
\Email{\href{mailto:marco.matassa@oslomet.no}{marco.matassa@oslomet.no}}

\ArticleDates{Received March 31, 2020, in final form September 25, 2020; Published online October 03, 2020}

\Abstract{We~study the twisted Hoch\-schild homology of~quantum flag manifolds, the twist being the modular automorphism of~the Haar state. We~prove that every quantum flag manifold admits a~non-trivial class in~degree two, with an~explicit representative defined in~terms of~a certain projection. The corresponding classical two-form, via the Hoch\-schild--Kostant--Rosenberg theorem, is identified with a~K\"ahler form on the flag manifold.}

\Keywords{quantum flag manifolds; twisted Hoch\-schild homology; K\"ahler forms}

\Classification{17B37; 20G42; 16E40}

\renewcommand{\thefootnote}{\arabic{footnote}}
\setcounter{footnote}{0}

\section{Introduction}

The study of~quantum homogeneous spaces corresponding to compact quantum groups is an~active area of~research.
In~this paper we study certain analogues of~differential forms on quantum flag manifolds, which constitute a~rich class of~quantum homogeneous spaces.

Let us first explain how \emph{Hoch\-schild homology} fits into the picture.
One possible approach to defining differential forms on non-commutative spaces is based on the \emph{Hoch\-schild--Kostant--Rosenberg} theorem~\cite{hkr}, or rather its continuous version. According to this theorem, the graded-commutative algebra of~differential forms $\Omega^\bullet(M)$ on a~smooth manifold $M$ is isomorphic to~$HH_\bullet(C^\infty(M))$, the Hoch\-schild homology of~the algebra $C^\infty(M)$ of~smooth functions on $M$.
Since $HH_\bullet(A)$ makes sense for any associative algebra $A$, it is a~candidate
for the role of~differential forms on the non-commutative space associated with $A$.

This works well for certain algebras, most notably the non-commutative tori~\cite{connes}, but it can be quite degenerate for other non-commutative spaces.
In~particular consider the compact quantum groups $\CqU$, with $U$ a~compact simple Lie group.
Their Hoch\-schild homology was computed by~Feng and Tsygan in~\cite{feng-tsygan}.
They find that the Hoch\-schild homological dimension of~$\CqU$ is equal to~$\mathrm{rank}(U)$, which is in~stark constrast with that of~$C^\infty(U)$ being $\mathrm{dim}(U)$.
This phenomenon is usually referred to as a~\emph{dimension drop}.

Eventually it was discovered that the dimension drop for~$\CqU$ can be avoided by~introducing an~appropriate \emph{twisting} in~Hoch\-schild homology.
This was first observed by~explicit computations for~$\bbC_q[SU_2]$ in~\cite{twisted-su2}, and later elucidated by~Brown and Zhang in~\cite{brzh}.
They also show that this twisted Hoch\-schild homology and the corresponding cohomology satisfy a~general version of~\emph{Poincaré duality}, as described by~Van den Bergh in~\cite{vandenbergh}.
The appropriate twisting turns out to be the \emph{modular automorphism} of~$\CqU$, which we will denote by~$\theta$, as shown by~Dolgushev in~the setting of~Poisson geometry~\cite{dolgushev}.

All of~the above can be considered also for \emph{quantum flag manifolds}.
These are certain $*$-sub\-algebras of~$\CqU$ denoted by~$\Cqflag$, with $U / K_S$ a~classical flag manifold.
By the preceding discussion the twisted Hoch\-schild homology $HH^\theta_\bullet\big(\Cqflag\big)$ is an~interesting object to study.
This has been completely determined for the quantum $2$-sphere in~\cite{hadfield}, but it is not known in~general.
In~\cite{twist-hoch} we have studied the case of~degree two for the quantum full flag manifolds $\Cqfullflag$, with the result that $HH^\theta_2\big(\Cqfullflag\big)$ has dimension at least $\mathrm{rank}(U)$, as well as providing explicit representatives for these classes.

An important motivation for studying the twisted Hoch\-schild homology of~$\Cqflag$ in~degree two is given by~\emph{K\"ahler forms}.
Indeed, classically the flag manifolds $U / K_S$ are K\"ahler and hence admit these two-forms with special properties.
The study of~analogues of~K\"ahler forms in~the quantum setting is important from the point of~view of~non-commutative complex and K\"ahler geometry, see for instance~\cite{besm13, obu16, obu17}.

We~now come to the main results of~this paper.
Our first result shows the non-triviality of~twisted Hoch\-schild homology in~degree two for any quantum flag manifold.

\begin{theorem*}
Given any quantum flag manifold $\Cqflag$, there exists a~non-trivial class $[C(\proj)] \in HH^\theta_2\big(\Cqflag\big)$ represented by~an~explicit element $C(\proj) \in \Cqflag^{\otimes 3}$.
\end{theorem*}

The explicit form involves a~certain projection $\proj$ with entries in~$\Cqflag$.
Next, the representative $C(\proj)$ admits an~appropriately defined classical limit for~$q \to 1$.
Under the Hoch\-schild--Kostant--Rosenberg theorem, this classical element corresponds to a~two-form on $U / K_S$.
Our second result shows that this is a~K\"ahler form.

\begin{theorem*}
Let $\omega_\proj \in \Omega^2(U / K_S)$ be the form corresponding to~$[C(\proj)]$ under the classical limit.
Then $\omega_\proj$ is a~K\"ahler form on $U / K_S$.
\end{theorem*}

We~will actually construct the K\"ahler form $\omega_\proj$ entirely in~the classical setting, in~a~way which is suitable for comparison with the quantum case.

Finally let us mention the connection with another common approach to differential forms on quantum spaces, due to Woronowicz~\cite{woronowicz}.
In~this approach, given an~algebra with an~action of~a compact quantum group, we introduce the structure of~a \emph{differential calculus} on the given algebra, together with various requirements about the action of~the compact quantum group.
In~general we have many inequivalent choices of~differential calculi, but the situation is much better for the quantum \emph{irreducible} flag manifolds $\Cqflag$.
In~this case it turns out that there is a~unique analogue of~the de Rham complex, which enjoys essentially all the classical properties, as shown by~Heckenberger and Kolb in~\cite{heko06}.
Within this setting, there is also a~notion of~K\"ahler forms introduced in~\cite{obu17}.
Quantum irreducible flag manifolds were shown to admit K\"ahler forms in~this sense in~\cite{mat-kahler}.
Finally we will show that these can also be identified with the forms $\omega_\proj$ in~the classical limit.

The structure of~the paper is as follows.
In~Section~\ref{sec:notations} we recall various preliminaries on Lie algebras and quantum groups and fix some notation.
In~Section~\ref{sec:hochschild} we summarize some results from~\cite{twist-hoch} on twisted Hoch\-schild homology of~quantum flag manifolds.
In~Section~\ref{sec:quantum-classes} we prove that all quantum flag manifolds admit non-trivial classes in~degree two.
In~Section~\ref{sec:kahler-forms} we introduce some K\"ahler forms on classical flag manifolds, from a~point of~view suitable for comparison with the quantum case.
Finally, in~Section~\ref{sec:comparison} we compare the classical limit of~the quantum classes with the K\"ahler forms obtained before.

\section{Notation and preliminaries}
\label{sec:notations}

In~this section we fix most of~our notation and recall various preliminary notions.
These include Lie algebras and parabolic subalgebras on the classical side, and quantized enveloping algebras and coordinate rings on the quantum side.
In~particular we try to adopt notation that illustrate the link between the two sides as much as possible.

\subsection{Lie algebras and Lie groups}

Let $\lieg$ be a~complex simple Lie algebra with fixed Cartan subalgebra $\lieh$.
We~denote by~$\Delta(\lieg)$ the roots of~$\lieg$ with respect to~$\lieh$, and by~$\Delta^\pm(\lieg)$ a~choice of~positive/negative roots.
We~write $\{\alpha_i\colon i \in I\}$ for the simple roots and $\{\omega_i\colon i \in I\}$ for the fundamental weights.
We~denote by~$(\cdot, \cdot)$ the bilinear form on $\lieh^*$ obtained from the Killing form, rescaled in~such a~way that $(\alpha, \alpha) = 2$ for the short roots.

Given $\alpha \,{\in}\, \Delta^+(\lieg)$, we choose root vectors $e_\alpha \,{\in}\, \lieg_\alpha$ and $f_\alpha \,{\in}\, \lieg_{-\alpha}$ normalized as in~\cite[The\-o\-rem~6.6]{knapp}.
In~particular we have $[e_\alpha, f_\alpha] = h_\alpha$ and $\alpha(h_\beta) = (\alpha, \beta)$.

We~denote by~$\lieu$ the \emph{compact real form} of~$\lieg$. It is given by
\[
\lieu = \bigoplus_{i \in I} \bbR \iu h_{\alpha_i} \oplus \bigoplus_{\alpha \in \Delta^+(\lieg)} \bbR (e_\alpha - f_\alpha) \oplus \bigoplus_{\alpha \in \Delta^+(\lieg)} \bbR \iu (e_\alpha + f_\alpha).
\]
Here and in~the following $\iu \in \bbC$ will denote the imaginary unit.

Corresponding to the Lie algebras $\lieg$ and $\lieu$, we have the connected, simply-connected Lie groups $G$ and $U$.
We~will also denote by~$T$ the maximal torus of~$U$.

\subsection{Parabolic subalgebras and flag manifolds}

Let $S$ be a~subset of~the simple roots, that is $S \subset I$.
Corresponding to this choice, we define the roots
\[
\Delta(\liel_S) := \mathrm{span} \{ \alpha_i\colon i \in S \} \cap \Delta(\lieg).
\]
In~terms of~these we define the \emph{Levi factor} corresponding to~$S$ by
\[
\liel_S := \lieh \oplus \bigoplus_{\alpha \in \Delta(\liel_S)} \lieg_\alpha.
\]
It is a~reductive Lie subalgebra containing the Cartan subalgebra.
We~also define
\[
\Delta(\lien_S^\pm) := \Delta^\pm(\lieg) \backslash \Delta^\pm(\liel_S).
\]
In~other words, the set $\Delta(\lien_S^+)$ contains the positive roots which are not in~$\Delta(\liel_S)$, and similarly for~$\Delta(\lien_S^-)$.
In~terms of~these roots we define
\[
\lien_S^\pm := \bigoplus_{\alpha \in \Delta(\lien_S^\pm)} \lieg_\alpha.
\]
These are also Lie subalgebras of~$\lieg$, which we call the positive and negative \emph{nilradical} corresponding to~$S$.
They are $\liel_S$-modules with respect to the adjoint action.

In~terms of~the previously defined subalgebras we have the decomposition
\[
\lieg = \lien_S^+ \oplus \liel_S \oplus \lien_S^-.
\]
We~call $\liep_S := \liel_S \oplus \lien_S^+$ the \emph{standard parabolic} subalgebra corresponding to~$S$.
Here standard parabolic means that $\liep_S$ is a~subalgebra containing the \emph{standard Borel} subalgebra
\[
\lieb := \lieh \oplus \bigoplus_{\alpha \in \Delta^+(\lieg)} \lieg_\alpha.
\]
In~the case $S = \varnothing$ we have $\liep_S = \lieb$.
We~also note that $\lieg / \liep_S \cong \lien_S^-$ as $\liel_S$-modules.

In~terms of~the compact real form $\lieu$, we will consider
\[
\liek_S := \liel_S \cap \lieu, \qquad
\liem_S := (\lien_S^+ \oplus \lien_S^-) \cap \lieu.
\]
We~have that $\liek_S$ is a~Lie subalgebra of~$\lieu$, but this is not the case for~$\liem_S$.

Corresponding to these Lie subalgebras, we have the (generalized) \emph{flag manifolds}.
These are homogeneous spaces of~the form $G / P_S$, where $P_S$ is the parabolic subgroup with Lie algebra $\liep_S$.
This definition shows that they are \emph{complex} manifolds.
It is also possible to give a~realization in~terms of~$U$, the compact real form of~$G$, by~the \emph{diffeomorphism}
\[
G / P_S \cong U / K_S, \qquad K_S := P_S \cap U.
\]
The realization $U / K_S$ makes it clear that they are \emph{compact} manifolds.
Finally there is also a~projective realization of~the flag manifolds, which we will describe in~Section~\ref{sec:quantum-classes}.

Some notable subclasses are the following.
For $S = \varnothing$ we have the homogeneous spaces $G / B \cong U / T$, called the \emph{full} flag manifolds.
At the other extreme we have the \emph{irreducible} flag manifolds, corresponding to the case $S = I \backslash \{t\}$ with $\alpha_t$ having multiplicity $1$ in~the highest root of~$\lieg$. These include the Grassmannians, for instance.

\subsection{Quantized enveloping algebras}

We~will use the conventions of~\cite{klsc}, which will be our main reference for this part.
Given $0 < q < 1$, the \emph{quantized enveloping algebra} $\Uqg$ is a~Hopf algebra deformation of~the enveloping algebra $U(\lieg)$ defined as follows.
It has generators $\{ K_i,\,E_i,\, F_i \}_{i = 1}^r$, with $r := \mathrm{rank}(\lieg)$, and relations as in~\cite[Section~6.1.2]{klsc}.
In~particular, the comultiplication, antipode and counit are given by
\begin{gather*}
\begin{aligned}
&\Delta(K_i) = K_i \otimes K_i,\quad
&&\Delta(E_i) = E_i \otimes K_i + 1 \otimes E_i,\quad
&&\Delta(F_i) = F_i \otimes 1 + K_i^{-1} \otimes F_i,
\\
&S(K_i) = K_i^{-1},
&&S(E_i) = - E_i K_i^{-1},
&&S(F_i) = - K_i F_i,
\\
&\varepsilon(K_i) = 1,
&&\varepsilon(E_i) = 0,
&&\varepsilon(F_i) = 0.
\end{aligned}
\end{gather*}
Given $\lambda = \sum_{i = 1}^r n_i \alpha_i$ we will write $K_\lambda := K_1^{n_1} \cdots K_r^{n_r}$.
Let $\rho := \frac{1}{2} \sum_{\alpha > 0} \alpha$ be the half-sum of~the positive roots of~$\lieg$.
Then we have $S^2(X) = K_{2 \rho} X K_{2 \rho}^{-1}$ for any $X \in \Uqg$.

We~will also consider a~$*$-structure on $\Uqg$, which in~the classical case corresponds to the compact real form $\lieu$.
We~can take for instance
\[
K_i^* = K_i, \qquad
E_i^* = K_i F_i, \qquad
F_i^* = E_i K_i^{-1}.
\]
The precise formulae are not so important here, as any equivalent $*$-structure will work equally well for our purposes.
We~will write $\Uqu := (\Uqg, *)$ when we consider $\Uqg$ endowed with the $*$-structure corresponding to the compact real form.

We~will also consider a~quantum analogue of~the Levi factor $\liel_S$, following~\cite{quantum-flag}.
The \emph{quantized Levi factor} $\UqlS$ is defined by
\[
\UqlS := \langle K_i, \, E_j, \, F_j \colon i \in I, \, j \in S \rangle \subset \Uqg.
\]
Here $\langle \cdot \rangle$ denotes the subalgebra generated by~the given elements in~$\Uqg$.
It is easily verified that $\UqlS$ is a~Hopf subalgebra.
Moreover it is a~Hopf $*$-subalgebra with $*$ corresponding to the compact real form.
Taking the $*$-structure into account we write $\UqkS := (\UqlS, *)$.
In~the special case $S = \varnothing$ we write instead $\Uqt := \langle K_i\colon i \in I \rangle$ with its $*$-structure.

\subsection{Quantized coordinate rings}

The \emph{quantized coordinate ring} $\CqG$ is defined as a~subspace of~the linear dual $\Uqg^*$.
We~take the span of~all the matrix coefficients of~the finite-dimensional irreducible representations $V(\lambda)$ (see below).
It becomes a~Hopf algebra by~duality in~the following manner: given $X$, $Y \in \Uqg$ and $a$, $b \in \CqG$ we define
\begin{gather*}
(a b)(X) := (a \otimes b) \Delta(X), \qquad
\Delta(a) (X \otimes Y) := a(X Y),
\\
S(a)(X) := a(S(X)), \qquad
1(X) := \varepsilon(X), \qquad
\varepsilon(a) := a(1).
\end{gather*}
Moreover it becomes a~Hopf $*$-algebra by~setting
\[
a^*(X) := \overline{a(S(X)^*)}.
\]
We~write $\CqU := (\CqG, *)$ for~$\CqG$ endowed with this $*$-structure.

We~have a~left action $\triangleright$ and a~right action $\triangleleft$ of~$\Uqg$ on $\CqG$ given by \[
(X \triangleright a)(Y) := a(Y X), \qquad
(a \triangleleft X)(Y) := a(X Y).
\]
Using the action of~$\Uqg$ on $\CqG$ we can define quantum analogues of~the (generalized) flag manifolds.
The \emph{quantum flag manifold} $\Cqflag$ is defined by
\[
\Cqflag := \{ a\in \CqU\colon X \triangleright a = \varepsilon(X) a, \, \forall\, X \in \UqkS \}.
\]
Notice that $\Cqflag = \CqU^{\UqkS}$, the invariants with respect to the action of~$\UqkS$.
In~the special case $S = \varnothing$ we will write $\Cqfullflag := \CqU^{\Uqt}$, as this case corresponds to the quantum analogue of~the full flag manifold $U / T$.

\subsection{Matrix coefficients}\label{sec:mat-coeff}

The representation theory of~$\Uqg$ is essentially the same as that of~$U(\lieg)$, hence of~$\lieg$.
In~particular we have the analogue of~the highest weight modules~$V(\lambda)$ for any dominant weight~$\lambda$, which we will denote by~the same symbol.
In~any case, given a~finite-dimensional representa\-tion~$V$, we define its \emph{matrix coefficients} by
\[
\big(c^V_{f, v}\big)(X) := f(X v), \qquad f \in V^*, \quad v \in V, \quad X \in \Uqg.
\]
These elements span $\CqG$, according to the description given above.

In~particular we will be interested in~unitary representations.
We~say that an~inner product $(\cdot, \cdot)$ on $V$ is \emph{$\Uqu$-invariant} if it satisfies
\[
(X v, w) = (v, X^* w), \qquad \forall\, v, w \in V, \quad \forall\, X \in \Uqu.
\]
Here we use the $*$-structure of~$\Uqu$.
It is well-known that an~$\Uqu$-invariant inner product exists on every representation $V(\lambda)$, and it is unique up to a~constant.
Let $\{v_i\}_i$ be an~orthonormal weight basis of~$V(\lambda)$ with respect to~$(\cdot, \cdot)$, and write $\lambda_i$ for the \emph{weight} of~$v_i$.
Also denote by~$\{f^i\}_i$ the dual basis of~$V(\lambda)^*$.
With this notation we set
\[
\big(u^{V(\lambda)}\big)^i_j := C^{V(\lambda)}_{f^i, v_j}.
\]
We~will often omit the superscript $V(\lambda)$ from the notation, as we will mostly work with one fixed representation in~the following.
Note that we have
\[
u^i_j(X) = f^i(X v_j) = (v_i, X v_j).
\]

\subsection{Modular automorphism}

The Hopf algebras $\CqU$ (or rather their completions as $C^*$-algebras) are examples of~\emph{compact quantum groups}.
According to the general theory, there is a~unique state $h\colon \CqU \to \bbC$, called the \emph{Haar state}, satisfying the properties
\[
(h \otimes \id) \circ \Delta(a) = h(a) = (\id \otimes h) \circ \Delta(a), \qquad \forall\, a \in \CqU.
\]
This state is not a~trace, but instead we have the property
\[
h(a b) = h(b \theta(a)), \qquad \forall \, a, b \in \CqU,
\]
where $\theta$ is the \emph{modular automorphism} corresponding to the Haar state.

The modular automorphism has a~simple expression in~terms of~the action of~$\Uqu$ on~$\CqU$.
Recall that $S^2(X)\, {=}\, K_{2 \rho} X K_{2 \rho}^{-1}$ for all $X \,{\in}\, \Uqu$, from which it follows that $S^2(a)\,{ =}\,K_{2 \rho}^{-1} \triangleright a\triangleleft K_{2 \rho}$ for all $a \in \CqU$.
By~\cite[Chapter 11, Proposition~34]{klsc} we have
\[
\theta(a) = K_{2 \rho} \triangleright a \triangleleft K_{2 \rho}, \qquad \forall\, a \in \CqU.
\]
In~particular, for the unitary matrix coefficients $u^i_j$ we have
\[
\theta(u^i_j) = q^{(2 \rho, \lambda_i + \lambda_j)} u^i_j,
\]
where $\lambda_i$ denotes the weight of~the basis element $v_i$.

\section{Some results on twisted Hoch\-schild homology}
\label{sec:hochschild}

In~this section we recall some basics of~(twisted) Hoch\-schild homology, as well as some results obtained in~\cite{twist-hoch}.
In~particular we discuss how to construct certain $2$-cycles on quantum flag manifolds in~terms of~appropriate projections.

\subsection{(Twisted) Hoch\-schild homology}

\emph{Hoch\-schild homology} is a~homology theory for associative algebras, which we consider here to be over $\mathbb{C}$.
The main reference for this section is~\cite{loday}.
Let $A$ be an~associative algebra and $M$ be an~$A$-bimodule.
Write $C_n (A, M) := M \otimes A^{\otimes n}$. The \emph{Hoch\-schild boundary} is the linear map $\mathrm{b}\colon C_n (A, M) \to C_{n - 1} (A, M)$ given by
\begin{gather*}
\mathrm{b} (m \otimes a_1 \otimes \cdots \otimes a_n) := m a_1 \otimes \cdots \otimes a_n
+ \sum_{i = 1}^{n - 1} (-1)^i m \otimes a_1 \otimes \cdots \otimes a_i a_{i + 1} \otimes \cdots \otimes a_n
 \\ \hphantom{\mathrm{b} (m \otimes a_1 \otimes \cdots \otimes a_n) :=}
{} + (-1)^n a_n m \otimes a_1 \otimes \cdots \otimes a_{n - 1}.
\end{gather*}
It satisfies $\mathrm{b}^2 = 0$, hence we have corresponding homology groups denoted by~$H_\bullet(A, M)$.
We~will also use the notation $HH_\bullet(A) := H_\bullet(A, A)$.
Hoch\-schild homology can also be defined in~terms of~derived functors as $H_n(A, M) = \mathrm{Tor}^{A^e}_n(A, M)$, where $A^e := A \otimes A^\mathrm{op}$.
There is a~corresponding dual cohomology theory, whose groups are denoted by~$H^n(A, M)$.

A natural choice of~bimodule is given by~$M = A$, in~which case we talk about \emph{the} Hoch\-schild homology of~$A$.
Here we will focus on the \emph{twisted bimodules} $M = {}_\sigma A$, which are defined as follows: as a~vector space we have $M = A$, but the bimodule structure is given by~$a \cdot b \cdot c = \sigma(a) b c $, where $\sigma \in \mathrm{Aut}(A)$.
In~this case we will use the notation $HH^\sigma_\bullet(A) := H_\bullet(A, {}_\sigma A)$ and refer to it as the \emph{twisted Hoch\-schild homology} of~$A$ (with twist $\sigma$).
Notice that we could also introduce a~twist for the right multiplication, but as bimodules this gives nothing new.

\begin{Remark}
The definition of~twisted Hoch\-schild homology $HH^\sigma_\bullet(A)$ presented here is essen\-ti\-ally as in~\cite{twisted-su2}, while the original cohomological setting appeared first in~\cite{kmt}.
There is also a~relation with the notion of~\emph{braided Hoch\-schild homology} introduced in~\cite{akrami-majid} and~\cite{baez}, for a~discussion of~these matters see~\cite[Example~3.9]{braided-homology}.
\end{Remark}

In~the case when $A$ is the algebra of~functions on some smooth space $X$, the Hoch\-schild homo\-logy $HH_\bullet(A)$ of~$A$ is related to the differential forms on the space $X$.
This is the \emph{Hoch\-schild--Kostant--Rosenberg} theorem~\cite{hkr}, see also~\cite[Theorem~3.4.4]{loday}.
Recall that for a~commutative unital algebra $A$, we have the $A$-module of~differential forms $\Omega^\bullet(A) := \bigwedge^\bullet_A \Omega^1(A)$ constructed from the module of~K\"ahler differentials $\Omega^1(A)$.

\begin{theorem*}[Hoch\-schild--Kostant--Rosenberg]
Let $A$ be a~commutative smooth algebra over $\bbC$.
Then there is an~isomorphism of~graded $\mathbb{C}$-algebras $HH_\bullet(A) \cong \Omega^\bullet(A)$.
\end{theorem*}

For the notion of~smooth algebra see~\cite[Section~3.4.1]{loday}.
The algebra structure on $HH_\bullet(A)$ is~given by~the \emph{shuffle product}, which relies on commutativity of~$A$, see~\cite[Section~4.2]{loday}.
There is also a~continuous version of~this theorem, essentially due to Connes~\cite{connes}, which allows to consider smooth forms as opposed to algebraic ones.

Finally we note that, at the level of~chains, the map $A^{\otimes n + 1} \to \Omega^n(A)$ is given by
\[
a_0 \otimes a_1 \otimes \cdots a_n \mapsto a_0 \diff a_1 \wedge \cdots \wedge \diff a_n.
\]

\subsection{Some results}

We~will now focus on the quantum flag manifolds $\Cqflag$.
We~will recall some results on their twisted Hoch\-schild homology and cohomology from~\cite{twist-hoch}.

Fix an~irreducible representation $V(\lambda)$ of~$\Uqu$ and write $N := \dim V(\lambda)$.
Let $\{v_i\}_{i = 1}^N$ be~an~ort\-ho\-normal weight basis with respect to an~$\Uqu$-invariant inner product, and write $\lambda_i$ for~the weight of~$v_i$.
Denote by~$u^i_j = (u^{V(\lambda)})^i_j$ the unitary matrix coefficients.

Denote by~$\mathrm{Mat}(\CqU)$ the set of~all matrices with entries in~$\CqU$.
Given $a$, $b \in \{ 1, \dots, N \}$, we define the $N \times N$-matrix $\matuni^a_b \in \mathrm{Mat}(\CqU)$ with entries
\[
(\matuni^b_a)^i_j := u^i_a \big(u^j_b\big)^*, \qquad
i, j \in \{ 1, \dots, N \}.
\]
The $*$-structure is extended to matrices by~the conjugate transpose, that is $(M^*)^i_j := \big(M^j_i\big)^*$ for~$M \in \mathrm{Mat}(\CqU)$.
We~also define the \emph{quantum trace} by
\[
\mathop{\rm Tr}\nolimits_q\big(\matuni^b_a\big) := \sum_{i = 1}^N q^{(2 \rho, \lambda_i)} \big(\matuni^b_a\big)^i_i.
\]
The matrices $\matuni^b_a$ behave like matrix units, as shown in~\cite[Proposition~3.3]{twist-hoch}.

\begin{Proposition}
\label{prop:matrix-units}
The matrices $\big\{\matuni^b_a\big\}_{a, b}$ are linearly independent and satisfy
\[
(\matuni^a_b)^* = \matuni^b_a, \qquad
\matuni^a_b \matuni^c_d = \delta^a_d \matuni^c_b, \qquad
\mathop{\rm Tr}\nolimits_q(\matuni^a_b) = \delta^a_b q^{(2 \rho, \lambda_a)}.
\]
\end{Proposition}

\begin{Remark}
These matrices are denoted by~$\mathsf{N}^b_a$ in~\cite{twist-hoch}.
In~the cited paper we also considered the matrices $\mathsf{M}^b_a$ given by~$(u^m_i)^* u^n_j$, but we will not use them here.
\end{Remark}

In~particular, the elements $\proj_a := \matuni^a_a$ are self-adjoint projections of~``quantum rank one".
We~will use them to construct certain twisted Hoch\-schild $2$-cycles on $\CqU$.

\begin{Remark}
It is worth noting that we have $\proj_a \in \mathrm{Mat}\big(\Cqfullflag\big)$, essentially by~construction.
On the other hand, we need additional conditions to get $\proj_a \in \mathrm{Mat}\big(\Cqflag\big)$.
\end{Remark}

Let us consider the elements $C(\proj_a) \in \CqU^{\otimes 3}$ defined by
\[
C(\proj_a) := \mathop{\rm Tr}\nolimits_q((2 \proj_a - 1) \otimes \proj_a \otimes \proj_a).
\]
Here the quantum trace is extended in~the obvious way, namely
\[
C(\proj_a) = \sum_{i, j, k = 1}^N q^{(2 \rho, \lambda_i)} (2 \proj_a - 1)^i_j \otimes (\proj_a)^j_k \otimes (\proj_a)^k_i.
\]
Recall that $\theta$ denotes the modular automorphism of~$\CqU$, which acts by~$\theta\big(u^i_j\big) = q^{(2 \rho, \lambda_i + \lambda_j)} u^i_j$.
The following result can be found in~\cite[Proposition~5.1]{twist-hoch}.

\begin{Proposition}
\label{prop:hoch-classes}
The element $C(\proj_a) \in \CqU^{\otimes 3}$ is a~$2$-cycle in~the $($normalized$)$ twisted Hoch\-schild complex, hence it defines a~class
\[
[C(\proj_a)] \in HH^\theta_2(\CqU).
\]
Moreover, if $(\proj_a)^i_j \in \Cqflag$ for all entries, then we also have a~class
\[
[C(\proj_a)] \in HH^\theta_2\big(\Cqflag\big).
\]
\end{Proposition}

\begin{Remark}
Here $C(\proj_a)$ is a~modification of~the usual Chern character $\mathrm{Ch}_n\colon K_0(A) \to H^\lambda_{2 n}(A)$ given by~$\mathrm{Ch}_n(P) = \mathop{\rm Tr}\big(P^{\otimes 2 n + 1}\big)$, where $H^\lambda_\bullet$ denotes the cyclic homology of~$A$.
This can be easily modified to the twisted case by~using the quantum trace, as opposed to the usual trace.
On the other hand, the factor $2 \proj_a - 1$ and the property $\mathop{\rm Tr}\nolimits_q(\proj_a) = q^{(2 \rho, \lambda_a)}$ guarantee that we map lands in~Hoch\-schild homology, as opposed to cyclic homology.
\end{Remark}

Next we would like to check whether the class $[C(\proj_a)]$ is non-trivial.
To do this we can int\-ro\-duce an~appropriate cohomology class and to show that the corresponding pairing is non-zero.
Given $a \in I$, consider the linear functional $\eta_a\colon \CqU^{\otimes 3} \to \bbC$ given by
\[
\eta_a(a_0 \otimes a_1 \otimes a_2) := \varepsilon(a_0)\, \varepsilon(F_a \triangleright a_1) \, \varepsilon(E_a \triangleright a_2).
\]
It is easy to check that, due to the properties of~the counit, these linear functionals define (twisted) cohomology classes, as shown in~\cite[Proposition~5.5]{twist-hoch}.

\begin{Proposition}
The restriction of~$\eta_a$ to~$\Cqflag$ gives a~cohomology class
\[
[\eta_a] \in HH^2_\theta\big(\Cqflag\big).
\]
\end{Proposition}

\begin{Remark}
Observe that this is true for the restriction of~$\eta_a$ to \emph{any} quantum flag manifold $\Cqflag$.
On the other hand, these functionals \emph{do not} give classes in~$HH^2_\theta(\CqU)$.
\end{Remark}

Finally we look at the pairing between $[\eta_a]$ and $[C(\proj_b)]$.
The result can be expressed entirely in~terms of~representation-theoretic data, as shown in~\cite[Proposition~6.11]{twist-hoch}.

\begin{Proposition}
\label{prop:pairing-classes}
The pairing between $[\eta_a]$ and $[C(\proj_b)]$ is given by
\[
\big\langle [\eta_a], [C(\proj_b)] \big\rangle = q^{(2 \rho - \alpha_a, \lambda_b)}
\big[(\alpha_a^\vee, \lambda_b)\big]_{q_a}.
\]
Here $\alpha^\vee = 2 \alpha / (\alpha, \alpha)$ is the coroot corresponding to~$\alpha$.
\end{Proposition}

Our aim in~the next section will be to produce some non-trivial classes in~$HH^\theta_2\big(\Cqflag\big)$.
In~order to do this, we will proceed in~two steps:
\begin{enumerate}\itemsep=0pt
\item[1)] define a~projection $\proj_1 \in \mathrm{Mat}\big(\Cqflag\big)$, giving a~class $[C(\proj_1)] \in HH^\theta_2\big(\Cqflag\big)$,
\item[2)] prove that it is non-trivial by~showing that $\langle [\eta_a], [C(\proj_1)]\rangle \neq 0$ for some $a \in I$.
\end{enumerate}

\section{Non-trivial classes on quantum flag manifolds}\label{sec:quantum-classes}

In~this section we will construct some non-trivial classes in~$HH^\theta_2\big(\Cqflag\big)$.
The first step will be to construct appropriate projections $\proj_1 \in \mathrm{Mat}\big(\Cqflag\big)$.
This will make use of~a certain irreducible representation, which in~the classical case is used to give a~projective realization of~$U / K_S$. With these projections at hand and the results of~the previous section, it will be fairly straightforward to show that we get non-trivial classes.

\subsection{The projections}
\label{sec:proj-quantum}

It is well-known that the flag manifold $U / K_S$ can be realized as a~$U$-orbit in~a~projective space, see for instance~\cite[Section~3.2.8]{cap-book}.
The projective space here is $\mathbb{P}(V(\rho_S))$, where $V(\rho_S)$ is the irreducible representation with highest weight
\[
\rho_S := \sum_{i \in I \backslash S} \omega_i.
\]
In~the quantum case we proceed along these lines by~considering the corresponding irreducible representation $V(\rho_S)$ of~$\Uqu$.
Let $\{v_i\}_i$ be an~orthonormal weight basis of~$V(\rho_S)$ with respect to~a~$\Uqu$-invariant inner product.
We~write $u^i_j = \big(u^{V(\rho_S)}\big)^i_j$ for the unitary matrix coefficients.
For~notational convenience, we will assume from now on that $v_1$ is a~\emph{highest weight vector} of~$V(\rho_S)$, hence of~corresponding weight $\lambda_1 = \rho_S$.

With the notation as above, we define the elements
\begin{equation}
\label{eq:quantum-proj}
p^i_j := u^i_1 \big(u^j_1\big)^* \in \CqU.
\end{equation}
Notice that $(\proj_1)^i_j = p^i_j$, using the notation of~the previous section.
Our goal will be to show that $p^i_j \in \Cqflag$.
First we will need the following lemma.

\begin{Lemma}
\label{lem:action-Fi}
We~have $F_i v_1 = 0$ for all $i \in S$.
\end{Lemma}

\begin{proof}
This works as in~the classical case, but we provide a~proof for completeness.
Suppose that $F_i v_1 \neq 0$, which implies that $F_i v_1$ has weight $\rho_S - \alpha_i$.
Recall that the Weyl group acts transitively on the weights of~an~irreducible representation.
Denoting by~$s_\alpha(\lambda) = \lambda - \frac{2 (\lambda, \alpha)}{(\alpha, \alpha)} \alpha$ the reflection of~$\lambda$ with respect to~$\alpha$, we find that $s_{\alpha_i}(\rho_S - \alpha_i) = s_{\alpha_i}(\rho_S) + \alpha_i$.
Moreover, since $(\rho_S, \alpha_i) = 0$ for~$i \in S$ by~definition of~$\rho_S$, we obtain $s_{\alpha_i}(\rho_S - \alpha_i) = \rho_S + \alpha_i$.
But this is impossible, since $\rho_S$ is the highest weight, hence we must have $F_i v_1 = 0$.
\end{proof}

We~are now ready to construct the invariant projections.

\begin{Proposition}
\label{prop:p-properties}
We~have $p^i_j \in \Cqflag$. Moreover we have
\[
\sum_k p^i_k p^k_j = p^i_j, \qquad
(p^i_j)^* = p^j_i, \qquad
\sum_i q^{(2 \rho, \lambda_i)} p^i_i = q^{(2 \rho, \rho_S)}.
\]
\end{Proposition}

\begin{proof}
Since $p^i_j = (\matuni^1_1)^i_j$, all claims follow from Proposition~\ref{prop:matrix-units} except for~$p^i_j \in \Cqflag$.
For~this it suffices to show that $p^i_j$ is invariant under the generators of~$\UqkS$.

Since $\big(X \triangleright u^i_1\big)(Y) = (v_i, Y X v_1)$, it is clear that $E_k \triangleright u^i_1 = 0$ for~$k \in S$ (this is true for any $k$, since $v_1$ is a~highest weight vector).
Next we have $F_k \triangleright u^i_1 = 0$ for~$k \in S$, by~Lemma~\ref{lem:action-Fi}.
On the other hand we have $K_k \triangleright u^i_1 = q^{(\rho_S, \alpha_k)} u^i_1$. Observe that $q^{(\rho_S, \alpha_k)} \neq 1$ for~$k \notin S$.

Now consider $\big(u^j_1\big)^*$. Since $\big(u^j_1\big)^*(X) = \overline{(v_j, S(X)^* v_1)}$ we have
\[
\big(X \triangleright \big(u^j_1\big)^*\big)(Y) = \overline{(v_j, S(Y X)^* v_1)} = \overline{(v_j, S(Y)^* S(X)^* v_1)}.
\]
We~easily conclude that $E_k \triangleright \big(u^j_1\big)^* = F_k \triangleright \big(u^j_1\big)^* = 0$ for~$k \in S$, as for the elements $u^i_1$.
On the other hand we have $K_i \triangleright \big(u^j_1\big)^* = q^{-(\rho_S, \alpha_k)} \big(u^j_1\big)^*$, since $S(K_k)^* = K_k^{-1}$.

Finally consider the elements $p^i_j$. Using $X \triangleright (a b) = (X_{(1)} \triangleright a) (X_{(2)} \triangleright b)$ it is clear that
\[
E_k \triangleright p^i_j = 0, \qquad
F_k \triangleright p^i_j = 0, \qquad
k \in S.
\]
On the other hand, using the results above, we have for any $k \in I$ that
\[
K_k \triangleright p^i_j = \big(K_k \triangleright u^i_1\big) \big(K_k \triangleright \big(u^j_1\big)^*\big) = p^i_j.
\]
Since $X \triangleright p^i_j = \varepsilon(X) p^i_j$ for the generators of~$\UqkS$, we have $p^i_j \in \Cqflag$.
\end{proof}

\begin{Remark}
It can be shown that the elements $p^i_j$ generate the quantum flag manifold $\Cqflag$, see~\cite[Proposition~3.2]{heko06}, but we will not need this fact.
\end{Remark}

\subsection{Non-triviality}

We~have just constructed a~projection $\proj_1$ with entries
\[
(\proj_1)^i_j = p^i_j = u^i_1 \big(u^j_1\big)^* \in \Cqflag.
\]
By Propositions~\ref{prop:hoch-classes} and~\ref{prop:p-properties}, we have a~corresponding class
\[
[C(\proj_1)] \in HH^\theta_2\big(\Cqflag\big).
\]
Finally we want to show that this class is non-trivial.

\begin{Theorem}
The class $[C(\proj_1)] \in HH^\theta_2\big(\Cqflag\big)$ is non-trivial.
\end{Theorem}

\begin{proof}
To show that the homology class $[C(\proj_1)] \in HH^\theta_2\big(\Cqflag\big)$ is non-trivial we will show that it has non-zero pairing with a~cohomology class $[\eta_a] \in HH^2_\theta\big(\Cqflag\big)$, using Proposition~\ref{prop:pairing-classes}.
Recall that for any $a \in I$ we have
\[
\big\langle [\eta_a], [C(\proj_1)] \big\rangle = q^{(2 \rho - \alpha_a, \rho_S)}
\big[(\alpha_a^\vee, \rho_S)\big]_{q_a},
\]
where we have used the fact that $\lambda_1 = \rho_S$.
Now for any $a \in I \backslash S$ we have $(\alpha_a^\vee, \rho_S) = 1$, since by~definition $\rho_S = \sum_{i \in I \backslash S} \omega_i$.
Hence for any $a \in I \backslash S$ we obtain
\[
\big\langle [\eta_a], [C(\proj_1)] \big\rangle = q^{(2 \rho - \alpha_a, \rho_S)} \neq 0.
\]
From this we conclude that $[C(\proj_1)]$ is non-trivial.
\end{proof}

Thus we have shown that, for every quantum flag manifold $\Cqflag$, we have a~non-trivial class $[C(\proj_1)] \in HH^\theta_2\big(\Cqflag\big)$.
These classes have an~appropriate classical limit for~$q \to 1$.
Then, according to the \emph{Hoch\-schild--Kostant--Rosenberg} theorem, they will correspond to some differential two-forms on $U / K_S$.
We~will investigate this aspect in~the following.

Recall that all flag manifolds $U / K_S$ are \emph{K\"ahler manifolds}.
In~particular, they admit two-forms with some special properties, the \emph{K\"ahler forms}.
In~the next section we will construct some K\"ahler forms on $U / K_S$ with the goal of~comparing them with the classes $[C(\proj_1)]$.

\section{K\"ahler forms on (classical) flag manifolds}
\label{sec:kahler-forms}

In~this section we will provide a~construction of~some K\"ahler forms on the flag manifolds $U / K_S$.
While there are many possible approaches, our aim is to proceed in~a~way that is well-suited for comparison with the quantum classes from the previous section.
The construction will use projections analogous to those used in~the quantum case.

\subsection{Notation}

Before getting into the construction, let us quickly explain the notation we will employ.
This will parallel what we have already used in~the quantum case.

Recall that the Lie groups $G$ and $U$, as well as the Lie algebras $\lieg$ and $\lieu$, all act in~a~compatible way in~a~given finite-dimensional representation $V$.
The representation of~$G$ will be \emph{holomorphic}, while the representation of~$U$ will be \emph{unitary} with respect to an~appropriate inner product.
We~will mainly consider the algebra of~\emph{matrix coefficients} $\CU \subset C^\infty(U, \bbC)$.
This algebra is spanned by~$c^V_{f, v}\colon U \to \bbC$ (for finite-dimensional representations) given by
\[
c^V_{f, v}(g) := f(g v), \qquad
f \in V^*, \quad v \in V, \quad g \in U.
\]
We~will occasionally consider the matrix coefficients $c^V_{f, v}$ as functions on $G$ according to the same formula.
In~the same way we can make sense of~$c^V_{f, v}(X)$ for~$X$ in~$\lieg$ or $\lieu$.

Now let $\{v_i\}_i$ be an~orthonormal weight basis with respect to the given inner product on~$V$.
Let us also denote by~$\{f^i\}_i$ the dual basis of~$V^*$. Then we have $c^V_{f^i, v_j}(g) = (v_i, g v_j)$.
Corresponding to this choice we will employ the notation $u^i_j := c^V_{f^i, v_j}$, that is
\[
u^i_j(g) = (v_i, g v_j).
\]
We~will omit the index $V$, as the representation $V$ will be fixed in~the following.
Recall that, given a~function $f\colon U \to \bbC$, its \emph{conjugate} $\overline{f}$ is defined by~$\overline{f}(g) := \overline{f(g)}$.
Since the matrix with entries $u^i_j(g)$ is unitary for every $g \in U$, we obtain the identity
\[
\overline{u^i_j}(g) = u^j_i\big(g^{-1}\big).
\]

Finally let us consider the action of~the Lie algebra $\lieg$ on $c^V_{f, v}$.
Considering $\lieg$ as derivations at $1 \in G$, it turns out that $X\big(c^V_{f, v}\big) = f(X v)$, where $\lieg$ acts on $V$ by~its representation.

\subsection{The projections}

In~this subsection we will construct a~projection $P$ for each flag manifold $U / K_S$, in~full analogy with the construction given in~Section~\ref{sec:proj-quantum}.

Corresponding to the subset $S \subset I$, we consider the dominant weight
\[
\rho_S = \sum_{i \in I \backslash S} \omega_i.
\]
We~have an~irreducible representation $V(\rho_S)$ of~highest weight $\rho_S$.
Let $\{v_i\}_i$ be an~orthonormal weight basis as above.
We~assume that $v_1$ is a~highest weight vector (of weight $\rho_S$).

As~we have already mentioned, the action of~$P_S$ preserves the line $\bbC v_1$, see for instance \mbox{\cite[Section~3.2.8]{cap-book}}.
This defines a~character $\charS\colon P_S \to \bbC^\times$ by
\[
g v_1 = \charS(g) v_1, \qquad g \in P_S.
\]
This restricts to a~character of~$K_S = P_S \cap U$, which we denote by~the same symbol.

\begin{Remark}
Recall that a~linear functional $\lambda\, {\in}\, \lieh^*$ is called \emph{analytically integral} (see~\cite[Sec\-tion~IV.7]{knapp}) if there is a~(multiplicative) character $\xi_\lambda\colon T \to \bbC^\times$ such that
\[
\xi_\lambda(\exp h) = e^{\lambda(h)}, \qquad \forall\, h \in \lieh \cap \lieu.
\]
Our notation $\charS$ for the character corresponding to~$V(\rho_S)$ is consistent with this one.
\end{Remark}

Now using the unitary matrix coefficients $u^i_j$ of~$V(\rho_S)$ we define
\begin{equation}
\label{eq:class-proj}
p^i_j := u^i_1 \overline{u^j_1} \in \CU \subset C^\infty(U, \bbC).
\end{equation}
We~will now derive some properties satisfied by~the functions $p^i_j$.

\begin{Lemma}
\label{lem:proj-mat}
We~have $p^i_j \in C^\infty(U / K_S, \bbC)$. Moreover we have the identities
\[
\sum_k p^i_k p^k_j = p^i_j, \qquad
\overline{p^i_j} = p^j_i, \qquad
\sum_i p^i_i = 1.
\]
\end{Lemma}

\begin{proof}
We~need to show that $p^i_j(g k) = p^i_j(g)$ for all $g \in U$ and $k \in K_S$. We~compute
\[
u^i_1(g k) = (v_i, g k v_1) = \charS(k) u^i_1(g).
\]
On the other hand, using $\overline{u^j_1}(g) = {u^1_j}(g^{-1})$ we compute
\[
\overline{u^j_1}(g k) = {u^1_j}\big(k^{-1} g^{-1}\big) = \big(v_1, k^{-1} g^{-1} v_j\big) = \big(k v_1, g^{-1} v_j\big) = \overline{\charS(k)} \overline{u^j_1}(g).
\]
Here we have used that the representation is unitary. Therefore
\[
p^i_j(g k) = u^i_1(g k) \overline{u^j_1}(g k) = |\charS(k)|^2 u^i_1(g) \overline{u^j_1}(g) = p^i_j(g).
\]
The other properties are easy to check.
\end{proof}

We~denote by~$P$ the matrix with entries $p^i_j$.
By the previous lemma, this is an~orthogonal projection of~rank $1$ with entries in~$C^\infty(U / K_S, \bbC)$.

\subsection{Line bundles}

In~this subsection we will interpret the functions $u^i_1$ as sections of~a line bundle over $G / P_S$ $\cong U / K_S$.
The material here will be used only tangentially in~the following, but it gives an~inte\-res\-ting geometric perspective, as well as connecting the construction we will use with other ways to obtain K\"ahler forms on flag manifolds.

Recall that the Lie group $G$ can be considered as a~\emph{principal $P_S$-bundle} over the flag manifold~$G / P_S$ (similarly for the compact description $U / K_S$).
Given a~representation of~$G$ on a~vector space $V$, we can form the \emph{associated vector bundle} $G \times_{P_S} V$.
The points of~this bundle are the equivalence classes of~$G \times V$ with respect to the relation
\[
(g, v) \sim \big(g p, p^{-1} v\big), \qquad
g \in G, \quad p \in P_S, \quad v \in V.
\]
This is a~\emph{holomorphic} vector bundle if the given representation is holomorphic.
The \emph{sections} of~this bundle can be identified with the functions $f\colon G \to V$ such that
\[
f(g p) = p^{-1} f(g), \qquad g \in G, \quad p \in P_S.
\]

Now let us consider the holomorphic line bundle
\[
L_{-\rho_S} := G \times_{P_S} \bbC,
\]
where $P_S$ acts on $\bbC$ by~$p \cdot z = \charS(p)^{-1} z$.
Observe that this is a~holomorphic representation, since the character $\xi_S$ comes from the holomorphic representation of~$G$ on $V(\rho_S)$.

\begin{Lemma}
The functions $u^i_1\colon G \to \bbC$ are holomorphic sections of~$L_{-\rho_S}$.
\end{Lemma}

\begin{proof}
It is clear that they are holomorphic functions, since $u^i_1(g) = (v_i, g v_1)$ for~$g \in G$ and the representation of~$G$ on $V(\rho_S)$ is holomorphic.
To show that they are sections of~$L_{-\rho_S}$ we~observe that $p v_1 = \charS(p) v_1$ implies $u^i_1(g p) = \charS(p) u^i_1(g)$ for~$p \in P_S$.
\end{proof}

\begin{Remark}
The minus sign in~the definition of~$L_{-\rho_S}$ (that is, using the character $\charS^{-1}$ instead of~$\charS$) is due to the fact that we take the \emph{positive} Borel subgroup, as opposed to the negative one, which is the more common choice when stating the Borel--Weil theorem.
For a~formulation using this convention see for instance~\cite[Theorem~7.58]{sepanski}.
\end{Remark}

It is possible to proceed along these lines to obtain a~K\"ahler form on $U / K_S$, as we will now sketch.
Equipping the line bundle $L_{-\rho_S}$ with a~connection, its curvature gives a~$(1, 1)$-form $\eta$ (a representative of~the \emph{first Chern class} of~$U / K_S$).
Then $\eta$ is K\"ahler if the line bundle $L_{-\rho_S}$ is \emph{positive}, which in~turn is equivalent to~$L_{-\rho_S}$ being \emph{ample} by~Kodaira's embedding theorem.
But it is known that $L_{-\rho_S}$ is ample by~results of~Borel--Weil, see for instance~\cite[Theorem~6.5]{snow} (keeping in~mind the opposite convention for the Borel subgroup).

However we will not proceed this way, since this description is not particularly well-suited for comparison with the quantum setting.
Instead, we will define the candidate K\"ahler form in~terms of~the projection $P$ introduced before, using the Chern character.

\subsection{Differential forms}

Let $P$ be the projection with entries $p^i_j \in C^\infty(U / K_S, \bbC)$ from \eqref{eq:class-proj}.
Corresponding to this projection, we define a~two-form on $U / K_S$ by
\[
\kafot := \mathop{\rm Tr}(P \cdot \diff P \wedge \diff P) = \sum_{i, j, k} p^i_j \diff p^j_k \wedge \diff p^k_i.
\]
Here and in~the following we will adopt some obvious matrix-type notation.
Our aim will be to show that $\tilde{\omega}$ is, up to a~constant, a~K\"ahler form on $U / K_S$.

Before getting into that, let us motivate this choice from a~suitably non-commutative point of~view.
By Lemma~\ref{lem:proj-mat} the matrix $P$ with entries $p^i_j \in C^\infty(U / K_S, \bbC)$ is a~projection of~rank one.
Hence we have a~projective $C^\infty(U / K_S, \bbC)$-module of~rank one which, according to the \emph{Serre--Swan} theorem, corresponds to a~complex line bundle over $U / K_S$.
Moreover this bundle admits a~Hermitian structure, due to the fact that $P$ is orthogonal.
More importantly, this line bundle admits a~natural connection defined in~terms of~the projection $P$, namely the \emph{Levi-Civita} one.
Its curvature can be computed using the \emph{Chern character} and coincides with the two-form~$\tilde{\omega}$ defined above, up to a~factor.
For more on this point of~view, see for instance~\cite[Chapter~1]{karoubi} and~\cite[Chapter~8]{loday}.

We~will now show some basic properties of~$\kafot$.
We~remark that the fact that it is closed is a~general result of~Chern--Weil theory, but we give a~short proof for completeness.

\begin{Lemma}
The two-form $\kafot$ on $U / K_S$ satisfies the properties:
\begin{enumerate}\itemsep=0pt
\item[$1)$] it is closed,
\item[$2)$] it is left $U$-invariant.
\end{enumerate}
\end{Lemma}

\begin{proof}
(1) The exterior derivative of~$\kafot$ is the $3$-form given by
\[
\diff \kafot = \sum_{i, j, k} \diff p^i_j \wedge \diff p^j_k \wedge \diff p^k_i = \mathop{\rm Tr}(\diff P \wedge \diff P \wedge \diff P).
\]
Now consider the element $J = 2 P - 1$, which satisfies $J^2 = 1$.
Applying $\diff$ to this identity we get $J \diff P = - \diff P J$.
Using these properties of~$J$, we can observe that
\[
\diff \kafot = \mathop{\rm Tr}\big(J^2 \diff P \wedge \diff P \wedge \diff P\big) = - \mathop{\rm Tr}(J \diff P \wedge \diff P \wedge \diff P J)
 = - \mathop{\rm Tr}(\diff P \wedge \diff P \wedge \diff P) = 0.
\]

(2) The left translation $L_g\colon U \to U$ given by~$L_g h = g h$ induces a~map $L_g\colon U / K_S \to U / K_S$, denoted by~the same symbol.
A form $\omega$ on $U / K_S$ is left $U$-invariant if $L_g^* \omega = \omega$ for every $g \in U$.
In~other words, for any $g \in U$ we must have $L_g^* \omega_{g h} = \omega_h$ for every $h \in U / K_S$.
For the matrix coefficients $u^i_j$, considered as functions on $U$, we have
\[
L_g^* u^i_j = \sum_k \pi(g)^i_k u^k_j, \qquad
L_g^* \overline{u^i_j} = \sum_k \pi\big(g^{-1}\big)^k_i \overline{u^k_j}.
\]
Here $\pi$ denotes the representation of~$U$ on $V(\rho_S)$.
From these identities we immediately obtain that the pullback of~the functions $p^i_j$ is given by
\[
L_g^* p^i_j = \sum_{k, l} \pi(g)^i_k \pi\big(g^{-1}\big)^l_j p^k_l.
\]
Using this fact, it is easy to check that we have the identity
\[
\sum_{i, j, k} L_g^* p^i_j \otimes L_g^* p^j_k \otimes L_g^* p^k_i = \sum_{i, j, k} p^i_j \otimes p^j_k \otimes p^k_i.
\]
Since pullbacks are compatible with the wedge product and commute with the exterior derivative, we conclude that $L_g^* \kafot_{g h} = \kafot_h$ and hence $\kafot$ is left $U$-invariant.
\end{proof}

\subsection{Complex decomposition}

In~this section we explore the consequences of~the $u^i_1$ being holomorphic sections of~a line bundle over $U / K_S$.
We~begin with a~simple lemma.

\begin{Lemma}
\label{lem:diff-ids}
We~have the identities
\begin{gather*}
P \cdot \del P = 0,
\qquad \del P \cdot P = \del P,
\qquad
 P \cdot \delbar P = \delbar P, \qquad
 \delbar P \cdot P = 0.
\end{gather*}
\end{Lemma}

\begin{proof}
These can be easily derived from Lemma~\ref{lem:proj-mat} together with $\delbar(u^i_1) = 0$ and $\del\big(\overline{u^j_1}\big) = 0$, where the last two identities follow from $u^i_1$ being holomorphic. For instance we have
\[
\del p^i_j = \big(\del u^i_1\big) \overline{u^j_1} = \sum_k \big(\del u^i_1\big) \overline{u^k_1} u^k_1 \overline{u^j_1}
= \sum_k \big(\del p^i_k\big) p^k_j. \tag*{\qed}
\]
\renewcommand{\qed}{}
\end{proof}

\begin{Remark}
There is an~exact analogue of~these identities for the quantum irreducible flag manifolds in~terms of~the Heckenberger--Kolb calculus, see~\cite[Lemma 5.2]{mat-kahler}.
\end{Remark}

We~will now rescale $\kafot$ by~setting
\begin{equation}
\label{eq:kahler-form}
\kafo := - \iu \kafot = - \iu \mathop{\rm Tr}(P \cdot \diff P \wedge \diff P).
\end{equation}
We~recall that $\iu \in \bbC$ denotes the imaginary unit. The main reason for this rescaling is to make $\kafo$ into a~\emph{real} form (and also positive definite, as we will see later on).

\begin{Lemma}
We~have that $\kafo$ is a~real $(1, 1)$-form. Moreover
\[
\kafo = \iu \mathop{\rm Tr}(\del P \wedge \delbar P) = \iu \sum_{i, j} \del p^i_j \wedge \delbar p^j_i.
\]
\end{Lemma}

\begin{proof}
Taking into account the identities from Lemma~\ref{lem:diff-ids} we rewrite
\begin{gather*}
\kafot = \mathop{\rm Tr}(P \cdot (\del P + \delbar P) \wedge \diff P) = \mathop{\rm Tr}(\delbar P \wedge \diff P) = \mathop{\rm Tr}(\delbar P \wedge \del P) + \mathop{\rm Tr}(\delbar P \wedge \delbar P) \\ \phantom{\kafot}
{} = \mathop{\rm Tr}(\delbar P \wedge \del P).
\end{gather*}
Since $\kafo = - \iu \kafot$, we conclude that $\kafo$ is a~$(1, 1)$-form with the claimed expression.
Next we show that $\kafo$ is real, that is $\overline{\kafo} = \kafo$.
Using the identity $\overline{p^i_j} = p^j_i$ we compute
\[
\overline{\kafo} = - \iu \sum_{i, j} \overline{\del p^i_j} \wedge \overline{\delbar p^j_i}
= - \iu \sum_{i, j} \delbar p^j_i \wedge \del p^i_j = \kafo. \tag*{\qed}
\]
\renewcommand{\qed}{}
\end{proof}

\subsection{K\"ahler forms}

The notion of~K\"ahler form on a~complex manifold is a~standard one.
The most convenient characterization for us is the following (as in~\cite[Lemma 3.1.7]{huybrechts}, for instance): a~\emph{K\"ahler form} on $X$ is a~closed real $(1, 1)$-form $\omega$ which is positive definite, that is locally of~the form $\omega = \frac{\iu}{2} \sum_{i, j} h_{i j} \mathrm{d} z_i \wedge \mathrm{d} \bar{z}_j$ with $h_{i j}(x)$ a~positive definite Hermitian matrix for all $x \in X$. This is strictly related to the notion of~K\"ahler metric on $X$.

Our goal is to show that $\kafo$ from \eqref{eq:kahler-form} is a~K\"ahler form on the flag manifold $U / K_S$.
So far we have shown that it is a~closed real $(1, 1)$-form on $U / K_S$, which is also $U$-invariant.
To conclude that it is a~K\"ahler form we still need to show that it is positive definite.

As~$\kafo$ is $U$-invariant, it suffices to show that it is positive definite at the origin $o := K_S$, namely the identity coset of~$U / K_S$.
Recall that, since $U / K_S \cong G / P_S$, we can identify the \emph{holomorphic} tangent space at $o$ with $\lieg / \liep_S \cong \lien_S^-$ and the \emph{anti-holomorphic} tangent space with $\lien_S^+$.
In~other words, the former is the span of~the root vectors $\{f_\alpha\}_{\alpha \in \homroot}$, while the latter is the span of~the root vectors $\{e_\alpha\}_{\alpha \in \homroot}$.
We~will denote by~$f_\alpha^\star$ and $e_\alpha^\star$ the corresponding dual elements in~their respective cotangent spaces.

From the above discussion and the fact that $\kafo$ is a~$(1, 1)$-form, it follows that it must take the following form at the origin
\[
\kafo_o = \iu \sum_{\alpha, \beta \in \homroot} c_{\alpha \beta} f_\alpha^\star \wedge e_\beta^\star, \qquad c_{\alpha \beta} \in \bbC.
\]
To show that $\kafo$ is positive definite it suffices to show that the matrix with entries $c_{\alpha \beta}$ is positive definite.
This is what we will check in~the following.

First we will require the following simple lemma.

\begin{Lemma}
Let $\alpha \in \homroot$. Then $v_\alpha = (\rho_S, \alpha)^{-1/2} f_\alpha v_1$ is a~vector of~norm $1$.
\end{Lemma}

\begin{proof}
Recall that $v_1$ is a~highest weight vector of~weight $\rho_S$. Hence we have
\[
e_\alpha f_\alpha v_1 = [e_\alpha, f_\alpha] v_1 = h_\alpha v_1 = (\rho_S, \alpha) v_1.
\]
Now we observe that, since $\rho_S = \sum_{i \in I \backslash S} \omega_i$ and $\alpha \in \homroot$, we have $(\rho_S, \alpha) > 0$.
Hence we can consider the vector $v_\alpha = (\rho_S, \alpha)^{-1/2} f_\alpha v_1$.
To show that it has norm $1$ we compute
\[
(v_\alpha, v_\alpha) = (\rho_S, \alpha)^{-1} (f_\alpha v_1, f_\alpha v_1)
= (\rho_S, \alpha)^{-1} (v_1, e_\alpha f_\alpha v_1) = 1. \tag*{\qed}
\]
\renewcommand{\qed}{}
\end{proof}

Up to now the choice of~the orthonormal weight basis $\{v_i\}_i$ for~$V(\rho_S)$ was arbitrary.
We~will now require that the basis contains all the vectors $v_\alpha$ for~$\alpha \in \homroot$ as in~the lemma above.
Notice that this requirement makes sense, as they all have norm $1$.

\begin{Theorem}
\label{thm:kahler-form}
We~have that $\kafo$ is a~K\"ahler form on $U / K_S$. Moreover
\[
\kafo_o = \iu \sum_{\alpha \in \homroot} (\rho_S, \alpha) f_\alpha^\star \wedge e_\alpha^\star.
\]
\end{Theorem}

\begin{proof}
We~have already shown that $\kafo$ is a~closed, real $(1, 1)$-form.
Hence it suffices to show that~$\kafo$ is positive definite (that is, the corresponding symmetric bilinear form is positive definite).
As~$\kafo$ is $U$-invariant, it suffices to show this at the origin of~$U / K_S$.
As~we have already discussed, the $(1, 1)$-form $\kafo$ at the origin has the following form
\[
\kafo_o = \iu \sum_{\alpha, \beta \in \homroot} c_{\alpha \beta} f_\alpha^\star \wedge e_\beta^\star.
\]
To show that $\kafo$ is positive definite it suffices to show that the matrix with entries $c_{\alpha \beta}$ is positive definite.
Observe that $c_{\alpha \beta} = - \iu \kafo_o(f_\alpha, e_\beta)$.

We~will now determine $c_{\alpha \beta}$ using the expression $\kafo = - \iu \sum_{i, j, k} p^i_j \diff p^j_k \wedge \diff p^k_i$.
We~will consider $o = K_S$ as being $1 \in U$, as the result will not depend on the chosen representative. Then
\[
c_{\alpha \beta} = - \sum_{i, j, k} p^i_j(1) f_\alpha\big(p^j_k\big) e_\beta\big(p^k_i\big)
= - \sum_{i, j} f_\alpha\big(p^i_j\big) e_\beta\big(p^j_i\big).
\]
In~the last step we have used that $p^i_j(1) = \delta^i_j$, which follows from $p^i_j(1) = (v_i, v_j)$.
Moreover, since $p^i_j = u^i_1 \overline{u^j_1}$ and $f_\alpha$ is a~derivation at the identity, we have
\[
f_\alpha\big(p^i_j\big) = f_\alpha(u^i_1) \overline{u^j_1}(1) + u^i_1(1) f_\alpha\big(\overline{u^j_1}\big)
= \delta^j_1 f_\alpha\big(u^i_1\big).
\]
Here we have used the fact that $f_\alpha\big(\overline{u^j_1}\big) = 0$ by~weight reasons.
Similar computations show that $e_\beta\big(p^j_i\big) = \delta^j_1 e_\beta\big(\overline{u^i_1}\big)$.
Therefore we obtain the expression
\[
c_{\alpha \beta} = - \sum_i f_\alpha\big(u^i_1\big) e_\beta\big(\overline{u^i_1}\big).
\]
First we will consider $f_\alpha\big(u^i_1\big) = (v_i, f_\alpha v_1)$.
Since $f_\alpha v_1$ is proportional to~$v_\alpha$, it is clear that this is zero unless $v_i = v_\alpha$ (as the chosen basis is orthonormal).
In~this case we have
\[
f_\alpha\big(u^i_1\big) = (v_\alpha, f_\alpha v_1) = (\rho_S, \alpha)^{1/2} (v_\alpha, v_\alpha) = (\rho_S, \alpha)^{1/2}.
\]
Similarly consider $e_\beta\big(\overline{u^i_1}\big) = - \overline{(v_i, f_\beta v_1)}$.
Again this is zero unless $v_i = v_\alpha$, in~which case $e_\beta\big(\overline{u^i_1}\big) = -(\rho_S, \alpha)^{1/2}$.
Hence we conclude that $c_{\alpha \beta} = 0$ for~$\alpha \neq \beta$, while
\[
c_{\alpha \alpha} = (\rho_S, \alpha), \qquad \alpha \in \homroot.
\]
Since $(\rho_S, \alpha) > 0$ for~$\alpha \in \homroot$, we conclude that $\kafo$ is positive definite.
\end{proof}

\begin{Remark}
This result is essentially due to Borel--Hirzebruch,~\cite[Proposition~14.6]{bo-hi}.
See also~\cite[Theorem~7.5]{snow}, taking into account different conventions.
\end{Remark}

\section{Comparison in~the classical limit}
\label{sec:comparison}

In~this last section we will consider an~appropriate classical limit of~the classes $[C(\proj_1)] \in HH^\theta_2\big(\Cqflag\big)$ constructed in~Section~\ref{sec:hochschild}.
We~will show that the corresponding classical two-forms can be identified with the K\"ahler forms on $U / K_S$ constructed in~Section~\ref{sec:kahler-forms}.

\subsection{Remarks on the classical limit}

Informally, the quantized enveloping algebra $\Uqu$ reduces to the enveloping algebra $U(\lieu)$ for~$q \to 1$, which we refer to as the \emph{classical limit}.
This specialization can be made precise with a~bit of~care, see for instance the approach in~\cite[Section~3.4]{hong-kang}.
For our purposes, the only thing we need to know is that the classical limit $q \to 1$ makes sense at the level of~representations.

Next, we want to consider the classical limit of~the quantized coordinate ring $\CqU$.
This was defined as the algebra of~matrix coefficients of~finite-dimensional $\Uqu$-repre\-sentations, hence for~$q \to 1$ it reduces to the algebra of~matrix coefficients of~finite-dimensional \mbox{$U(\lieu)$-repre}\-sentations.
Finally we can identify this algebra with $\CU \subset C^\infty(U, \bbC)$, the algebra of~mat\-rix coefficients of~finite-dimensional $U$-representations.
This is simply done by
\[
c^V_{f, v}(X) = f(X v) \longleftrightarrow c^V_{f, v}(g) = f(g v),
\]
where $X \in U(\lieu)$ and $g \in U$.
Observe that this identification makes sense, since $U$ and $\lieu$ act in~a~compatible way in~each finite-dimensional representation $V$.

\subsection{The comparison}

Consider now the representative $C(\proj_1) \in \Cqflag^{\otimes 3}$ of~the class $[C(\proj_1)] \in HH^\theta_2\big(\Cqflag\big)$.
It corresponds to an~element $\widetilde{C(\proj_1)} \in \bbC[U / K_S]^{\otimes 3}$ under the classical limit, as explained above.
Denote by~$\omega_{\proj_1} \in \Omega^2(U / K_S)$ the corresponding two-form obtained from the Hoch\-schild--Kostant--Rosenberg map.

\begin{Theorem}
The $($rescaled$)$ form $\omega_{\proj_1} \in \Omega^2(U / K_S)$ is a~K\"ahler form.
\end{Theorem}

\begin{proof}
The representative $C(\proj_1) \in \Cqflag^{\otimes 3}$ is given explicitly by
\[
C(\proj_1) = \sum_{i, j, k} q^{(2 \rho, \lambda_i)} \big(2 p^i_j - \delta^i_j\big) \otimes p^j_k \otimes p^k_i.
\]
The elements $p^i_j \in \CqU$ from \eqref{eq:quantum-proj} correspond to the elements $p^i_j \in \CU$ from \eqref{eq:class-proj} under the classical limit.
Hence we obtain
\[
\widetilde{C(\proj_1)} = \sum_{i, j, k} \big(2 p^i_j - \delta^i_j\big) \otimes p^j_k \otimes p^k_i \in \bbC[U / K_S]^{\otimes 3}.
\]
Now $\omega_{\proj_1}$ is the image of~$\widetilde{C(\proj_1)}$ under the Hoch\-schild--Kostant--Rosenberg map, which is given by~$a_0 \otimes a_1 \otimes \cdots a_n \mapsto a_0 \diff a_1 \wedge \cdots \wedge \diff a_n$. We~obtain the two-form
\[
\omega_{\proj_1} = \sum_{i, j, k} \big(2 p^i_j - \delta^i_j\big) \diff p^j_k \wedge \diff p^k_i \in \Omega^2(U / K_S).
\]
Observe that the second term is zero by~graded-commutativity, since
\[
\sum_{i, j} \diff p^i_j \wedge \diff p^j_i = - \sum_{i, j} \diff p^j_i \wedge \diff p^i_j = 0.
\]
Hence $\omega_{\proj_1}$ can be rewritten as
\[
\omega_{\proj_1} = 2 \sum_{i, j, k} p^i_j \diff p^j_k \wedge \diff p^k_i \in \Omega^2(U / K_S).
\]
This coincides with $\kafo$ as defined in~\eqref{eq:kahler-form}, up to a~constant.
But we know that the latter is a~K\"ahler form on $U / K_S$ according to Theorem~\ref{thm:kahler-form}, which gives the claim.
\end{proof}

\begin{Remark}
For a~quantum \emph{irreducible} flag manifold $\Cqflag$, we have constructed in~\cite{mat-kahler} a~K\"ahler form in~the sense of~\cite{obu17} using the differential calculus of~\cite{heko06}.
It is of~the form $\sum_{i, j, k} q^{(2 \rho, \lambda_i)} p^i_j \diff p^j_k \wedge p^k_i$, up to a~constant.
Hence, by~the same argument as above, it can be identified with a~K\"ahler form on the irreducible flag manifold $U / K_S$.
\end{Remark}

To summarize, the non-trivial classes $[C(\proj_1)] \in HH^\theta_2\big(\Cqflag\big)$ constructed in~this paper correspond to some K\"ahler forms on $U / K_S$, under an~appropriate classical limit.
For this reason we believe they will play an~important role in~the investigation of~the non-commutative complex and K\"ahler geometry of~the quantum flag manifolds $\Cqflag$.

\pdfbookmark[1]{References}{ref}
\LastPageEnding

\end{document}